\documentclass[a4paper,draft]{amsproc}
\usepackage{amssymb}
\usepackage[hyphens]{url} 

\theoremstyle{plain}
 \newtheorem{thm}{Theorem}[section]
 
 \newtheorem{lem}{Lemma}[section]
 
\theoremstyle{definition}

\theoremstyle{remark}
 
 \numberwithin{equation}{section}

\renewcommand{\leq}{\leqslant}
\renewcommand{\geq}{\geqslant}

\title[Existence    Uniqueness ]{On the Existence    Uniqueness  and Numerical Computation of    Non-linear Coupled Elliptic PDE System with its application.}

\subjclass[2010]{Primary REQUIRED; Secondary OPTIONAL}


\author[Surname]{\bfseries B.V. Rathish Kumar, Sangita Dey} 

\address{ 
Indian Institute of Technology Kanpur \\ Kanpur,U.P,208016,India 
}
\email{ bvrk@iitk.ac.in, maths.sangita@gmail.com}

\begin{document}

{\begin{flushleft}\baselineskip9pt\scriptsize
MANUSCRIPT
\end{flushleft}}
\vspace{18mm} \setcounter{page}{1} \thispagestyle{empty}

\begin{abstract}
In this study we prove the existence-uniqueness of a coupled
non-linear elliptic PDE system using Lax-Milgram theorem, Galerkin
Method, Brouwer's fixed point theorem. Later we derive the finite
element scheme for the numerical solution of the PDE system and also
carry out the convergence analysis for the derived scheme.  Further
successfully apply the scheme to an application related to free
convection phenomena.

\end{abstract}

\maketitle

\section{Introduction}  

Nonlinear analysis has become a topic of intense research for its
ability to deal with scientific and engineering applications. Study
of nonlinear elliptic PDE systems is one such topic in this regard.
Even though a good amount of research has been carried out on this
topic, still there is no general applicable method to prove
existence uniqueness of such problems. In some cases we can apply
few of the standard methods like variational
method\cite{Kes88}, semigroup
principle\cite{Kes88}, monotone
method\cite{Kes88}, Galerkin method
e.t.c.\cite{Kes88},\cite{Cer12,Chm86,Gir12,Gir09},\cite{Xsh05}.
In this paper we have used Galerkin approach to establish existence
uniqueness of a non-linear coupled elliptic PDE system
\begin{equation} 
-\triangle \psi -R_a\frac{\partial \theta}{\partial x}=f_1 ~~\mbox{in}~ \Omega \subset \mathbb{R}^2
\end{equation}
\begin{equation}
J(\psi,\theta) = \triangle \theta+f_2,f_2>0 ~\mbox{in}~ \Omega \subset \mathbb{R}^2
\end{equation}
where J denotes determinant of $Jacobian~
of~\psi,\theta~\mbox{and}~\psi,\theta \in ~ C^2(\Omega)$ are steam
function, and temperature respectively, $f_1,f_2 \in
~L^2(\Omega),~\mbox{the~ source~ term~},\\ \Omega=(0,1)^2,~ R_a
~\mbox{is~Rayleigh~ number}, ~\psi ,\theta = 0 ~\mbox{on}~ \partial
\Omega,$ which is known to describe the  natural convection and flow
dynamics in porous media. While a good number of analytical and
numerical works, related to its solution, have been reported in the
literature  the basic question of
 mathematical existence-uniqueness of the model has not been reported in anywhere  to the best on our knowledge.
Further this model is known to be basic model based on which several
other elliptic PDE models, viz., Darcy Brinkman model, Dary
Forchheimer model etc.,  describing more complex physical phenomena
have been proposed. Hence it is important to establish the
existence-uniqueness of the solution to the considered model.
 We have stated our main theorem in Section $2$ and established its proof by using Galerkin method and  Brouwer's fixed point theorem, under some conditions and assumptions mentioned therein. Section $3$ contains the finite element formulation and finite element analysis of the above mentioned PDE where we tried to derive the order of convergence simply by using Galerkin method. Section $4$ is about the numerical experiment and algorithm used for finite element computation where we have shown the contours representing the finite element solution for different values of $R_a$.
\\
\textbf{Definition:} $H^1(\Omega)=\{v~\in L^2(\Omega):\frac{\partial v}{\partial x},\frac{\partial v}{\partial y}~\in~L^2(\Omega)\}$,\\
$H_0^1(\Omega)=\{v~\in~H^1(\Omega):v=0~on~\partial\Omega\}$ where $\partial\Omega $ means boundary of $\Omega$.

\section{Our Model}\label{sec:2}
\label{}
The existence uniqueness of the coupled PDE system $(1)-(2)$ is established in this section through the following theorem.
\begin{thm}
Consider the problem P
\begin{equation}
-\triangle\psi -R_a\frac{\partial \theta}{\partial x}=f_1 \nonumber~~~~~ ~~~~~~~~~~~~~~~~~~~~~~~(1)
\end{equation}
\begin{equation}
J(\psi,\theta) = \triangle\theta+f_2, f_2>0,f_1,f_2~\in ~L^2(\Omega)\nonumber~~~~~(2)
\end{equation}
$\psi,\theta \in ~ C^2(\Omega),~\psi ,\theta = 0 ~\mbox{on}~ \partial \Omega$.
Problem P has atleast  one weak solution $(\psi,\theta) \in (H_0^1(\Omega))^2$, satisfying $\|\nabla \psi\|_2^2+\|\nabla \theta\|_2^2 \leq R^2, R^2=C^2 (\|f_1\|_2^2+\|f\|_2^2)/2B$
where the data satisfy the estimate $\|\frac{\partial \theta}{\partial x}\|_4 \leq L,\|\frac{\partial \theta}{\partial y}\|_4 \leq L,$
\\$B=\mbox{min}\{ (1-CR_a/2-AL),(1/2-R_a/2-AL)\}>0$. Furthermore the solution is unique under the conditions
$(1/2-CR/(\surd{2})-CR_a/2)~\mbox{and}~(1/2-R_a/2-{2\surd{2}RC})~ > 0$ where C is the constants due to Poincare inequality and A is the constant due to Sobolev embedding.
\end{thm}
\begin{proof}
\textbf{The Weak Formulation:}\begin{equation}\int_{\Omega}\nabla \psi \cdot\nabla v_1- R_a\int_{\Omega}(\frac{\partial \theta}{\partial x})v_1 + \int_{\Omega}\nabla \theta\cdot\nabla v_2 + \int_\Omega  J v_2 - \int_{\Omega}(f_1 v_1+f_2 v_2)=0\nonumber\end{equation}where $v_1,v_2~\in~H_0^1(\Omega), \Omega=(0,1)\times (0,1)$.
Since $(H_{0}^1(\Omega))^2$ is separable it has a countable orthonormal basis $\{e_1,e_2,e_3,......\}$ where \\$e_i=(e_i^1,e_i^2)~\mbox{is a sequence of smooth functions with}~e_i^2 ~\in~H_0^1(\Omega)\cap H^2(\Omega)~\forall~i$.\\Let $W_m=\mbox{span}\{e_1,e_2,e_3,.....,e_m\}~\forall ~m.$ Then the Galerkin formulation for the above problem on $W_m$ is
$:~ \mbox{to~ find} ~(\psi_m,\theta_m)~\mbox{such ~that}~ \forall ~1 \leq i \leq m, \\\int_{\Omega}\nabla \psi_m\cdot\nabla e_i^1- R_a\int_{\Omega}(\frac{\partial \theta}{\partial x})e_i^1 + \int_{\Omega}\nabla \theta\cdot\nabla e_i^2 + \int_\Omega  J e_i^2 - \int_{\Omega}(f_{1} e_i^2+f_{2} e_i^2)=0$. Denote this problem as $P_m$.\\
 \textbf{Brouwer's Fixed Point Theorem:} Let $T: B \rightarrow B $ be a continuous map where B is a compact, convex set. Then T has a fixed  point in B.
\begin{lem} We assume that the data satisfy the estimate as mentioned in Theorem $1$.
 Then the problem $(P_m)$ for each m has atleast one solution $(\psi_m,\theta_m) \in W_m$ satisfying, $\|\bigtriangledown \psi_m\|_2^2 + \|\bigtriangledown \theta_m\|_2^2 \leq R^2, R^2=C^2 (\|f_1\|_2^2+\|f\|_2^2)/2B$ and $B=\mbox{min}\{ (1-CR_a/2-AL),(1/2-R_a/2-AL)\}>0$.
\end{lem}
\begin{proof}Since $W_m$ is a Hilbert space of finite dimension we use Brouwer's fixed point theorem.To this end we introduce the map
\cite{Gir09} $F_m:W_m \rightarrow W_m$, defined for all $(\psi,\theta)~in~ W_m~$ by the following: for all $(\psi,\theta)\in W_m$
\begin{equation}
(F_m(\psi,\theta),(v_1,v_2))=\int_{\Omega}\nabla \psi\cdot\nabla v_1- R_a\int_{\Omega}(\frac{\partial \theta}{\partial x})v_1 + \int_{\Omega}\nabla \theta \cdot \nabla v_2 + \int_\Omega J v_2 - \int_{\Omega}(f_{1}v_1+f_{2}v_2)=0
{\nonumber}\end{equation}
$,~\forall (\psi,\theta)~\in W_m.$ Clearly $F_m $ defines a map from $W_m$ onto itself. As $W_m$ is finite dimensional this map is continuous. Furthermore, any zero of $F_m$ is a solution of problem $P_m$.
Let us evaluate $F_m((\psi,\theta),(\psi,\theta))~for ~ all ~(\psi,\theta) \in W_m ~ satisfying$\\
$\|\bigtriangledown \psi\|_2^2+\| \bigtriangledown \theta\|_2^2=R^2 ~\mbox{where}~
R^2=\frac{C^2 \| f _1\|_2^2+\| f _2\|_2^2}{2B}$.
\begin{equation}
Thus~ F_m((\psi,\theta),(\psi,\theta))
=\|\nabla \psi\|_2^2+\|\nabla \theta\|_2^2-R_a\int_{\Omega}(\frac{\partial \theta}{\partial x})\psi+\int_{\Omega} J (\psi,\theta)\theta -\int_{\Omega}(f_{1}\psi+f_{2}\theta ){\nonumber}\end{equation}
\begin{equation}
 \geq \|\nabla \psi\|_2^2 + \|\nabla \theta\|_2^2 - R_a \int_\Omega \mid(\frac{\partial \theta}{\partial x})\psi\mid -\int_\Omega \mid J(\psi,\theta)\theta \mid- \int_{\Omega}(f_{1}\psi+f_{2}\theta)\\
 {\nonumber}\end{equation}
\begin{equation}
 \geq \|\nabla \psi\|_2^2 + \|\nabla \theta\|_2^2 - R_a \int_\Omega \mid(\frac{\partial \theta}{\partial x})\psi\mid -\int_\Omega \mid \big[\frac{\partial \psi}{\partial x}\frac{\partial \theta}{\partial y}-\frac{\partial \psi}{\partial y}\frac{\partial \theta}{\partial x}]\theta \mid -\int_{\Omega}(f_{1}\psi+f_{2}\theta){\nonumber}\end{equation}
\begin{equation}
 \geq\|\nabla \psi\|_2^2 + \|\nabla \theta\|_2^2 - C\frac{R_a}{2} \|\nabla \psi \|_2^2 -\frac{R_a}{2} \|\nabla \theta \|_2^2 -A L (\|\nabla \psi\|_2^2+ \| \nabla \theta \|_2^2)-C(\|f_1\|_2\|\nabla \psi\|_2+\|f_2\|_2\|\nabla \theta \|_2){\nonumber}\end{equation}
 \begin{equation}\big[\int_\Omega \mid \frac{\partial \psi}{\partial x}\frac{\partial \theta}{\partial y}\theta \mid ~\leq~\|\frac{\partial \psi}{\partial x}\|_2 \|\frac{\partial \theta}{\partial y}\|_4\|\theta\|_4~\leq~LA\|\nabla \psi\|_2 \| \nabla \theta\|_2~\leq~\frac{LA}{2}(\|\nabla \psi \|_2^2+\|\nabla \theta\|_2^2)]{\nonumber}\end{equation}\\$\big[\mbox{By using the condition in theorem1 Sobolev embedding and Holder's inquality}\big]$\\
 Similarly, \begin{equation}\int_\Omega \mid \frac{\partial \psi}{\partial x}\frac{\partial \theta}{\partial y}\theta \mid ~\leq~\frac{LA}{2}(\|\nabla \psi \|_2^2+\|\nabla \theta\|_2^2) {\nonumber}\end{equation}
 \begin{equation}
 So,F_m((\psi,\theta),(\psi,\theta))~ \geq~ (1/2-CR_a/2-AL)\|\nabla  \psi\|_2^2+(1/2-R_a-AL)\|\nabla \theta\|_2^2-\frac{C^2}{2}(\|f_1\|_2^2+\|f_2\|_2^2){\nonumber}\end{equation}
  \begin{equation} \geq B(\|\nabla  \psi\|_2^2+\|\nabla \theta\|_2^2)-\frac{C^2}{2}(\|f_1\|_2^2+{\|f_2\|_2^2})=0
{\nonumber}\end{equation}
on the surface of the sphere centered at origin with radius R where\\ $R^2=\frac{C^2}{2B}(\|f_1\|_2^2+{\|f_2\|_2^2})$.(By Poincare inequality, Sobolev embedding, triangle inequality , and Young's inequality).\\ Then classical variant of Brouwer's fixed point theorem $ F_m $ has atleast one zero in this ball. This yields existence of the solution
$(\psi_m,\theta_m)$ of $P_m.$
\end{proof}
\begin{lem} Our problem P has atleast one solution $(\psi,\theta)~ \in~ (H_0^1(\Omega))^2$ satisfying the estimate $\|\nabla \psi\|_2^2+\|\nabla \theta\|_2^2 \leq R^2$ when the data satisfy the estimate as mentioned in Theorem $1$.\end{lem}
\begin{proof}
$(\psi_m,\theta_m)$  obtained from Theorem 1 is uniformly bounded by R and hence has a weakly convergent sub-sequence converging to $(\psi,\theta), say.$ W.L.O.G assume that $(\psi_m,\theta_m)$ converges weakly to $(\psi,\theta)$. Let $v=(v^1,v^2) \in (H_0^1(\Omega))^2$. Then there exists $\{e_m\} ~\in~W_m~such ~that~e_m \rightarrow v~ strongly~in ~H_0^1(\Omega)~and~L_2(\Omega). $\\
\begin{equation}
CLAIM: \lim_{m \to \infty}[\int_{\Omega}\nabla \psi_m\cdot\nabla e_m^1- R_a\int_{\Omega}(\frac{\partial \theta}{\partial x})e_m^1 + \int_{\Omega}\nabla \theta
\cdot \nabla e_m^2 + \int_\Omega \mid J \mid e_m^2 - \int_{\Omega}(f_1 e_m^1+f_2 e_m^2)]
{\nonumber}\end{equation}
\begin{equation}
=\int_{\Omega}\nabla \psi\cdot\nabla v^1- R_a\int_{\Omega}(\frac{\partial \theta}{\partial x})v^1 + \int_{\Omega}\nabla \theta \cdot \nabla v^2 + \int_\Omega \mid J \mid v^2 - \int_{\Omega}(f_1v^1+f_2v^2).
{\nonumber}\end{equation}
 PASSING TO THE LIMIT:It is very easy to pass to the limit for the linear terms.\\
{Dealing with non-linear term}:\begin{equation}\int_{\Omega}[J (\psi_m,\theta_m)v_m^2 -J (\psi,\theta)v^2]{\nonumber}\end{equation}
\begin{equation}
 =\int_{\Omega}[\frac{\partial\psi_m}{\partial x} \frac{\partial \theta_m}{\partial y}v_m^2-\frac{\partial\psi_m}{\partial y} \frac{\partial \theta_m}{\partial x}v_m^2-\frac{\partial\psi}{\partial x} \frac{\partial \theta}{\partial y}v^2+\frac{\partial\psi}{\partial y} \frac{\partial \theta}{\partial x}v^2].{\nonumber}\end{equation}
\begin{equation}
Now [\frac{\partial \psi_m}{\partial x}\frac{\partial \theta_m}{\partial y}v_m^2-\frac{\partial \psi}{\partial x}\frac{\partial \theta}{\partial y}v^2]
  =\frac{\partial \psi_m}{\partial x}\frac{\partial \theta_m}{\partial y}v_m^2-
 \frac{\partial \psi_m}{\partial x}\frac{\partial \theta}{\partial y}v^2+\frac{\partial \psi_m}{\partial x}\frac{\partial \theta}{\partial y}v^2-\frac{\partial \psi}{\partial x}\frac{\partial \theta}{\partial y}v^2{\nonumber}\end{equation}
\begin{equation}
 =\frac{\partial \psi_m}{\partial x}[\frac{\partial \theta_m}{\partial y}v_m^2-\frac{\partial \theta}{\partial y}v^2]+
 \frac{\partial \theta}{\partial y}v^2(\frac{\partial \psi_m}{\partial x}-\frac{\partial \psi}{\partial x})].{\nonumber}\end{equation}
 \begin{equation}
 \mid \int_{\Omega}\frac{\partial \theta}{\partial y}v^2(\frac{\partial \psi_m}{\partial x}-\frac{\partial \psi}{\partial x})\mid \leq \mid v^2 \mid _{\infty}\| \frac{\partial}{\partial x}(\frac{\partial \theta}{\partial y}) \|_2\|\psi_m-\psi\|_2 \rightarrow 0 ~as~  m \rightarrow \infty.{\nonumber}\end{equation}
\begin{equation}
 \frac{\partial \theta_m}{\partial y}v_m^2-\frac{\partial \theta}{\partial y}v^2=\frac{\partial \theta_m}{\partial y}v_m^2-\frac{\partial \theta_m}{\partial y}v^2+\frac{\partial \theta_m}{\partial y}v^2-\frac{\partial \theta}{\partial y}v^2=\frac{\partial \theta_m}{\partial y}(v_m^2-v^2) +(\frac{\partial \theta_m}{\partial y}-\frac{\partial \theta}{\partial y})v^2{\nonumber}\end{equation}
Our obtained solution lies in the ball of radius $R$ by Lemma $1$. $\Rightarrow ~\|\nabla \psi\|_2~ \leq R$
\begin{equation}
\Rightarrow \int_{\Omega}(\mid{\frac{\partial \psi}{\partial x}}\mid^2+\mid{\frac{\partial \psi}{\partial y}}\mid^2 )\leq R^2=\int_{\Omega}R^2 d \mu
{\nonumber}\end{equation}
\begin{equation}
 \Rightarrow \mid \frac{\partial \psi}{\partial x}\mid^2  \leq R^2/2~a.e.,~\mid \frac{\partial \psi}{\partial y}\mid^2 \leq R^2/{2}~a.e.
 {\nonumber}\end{equation}
 \begin{equation}
 \Rightarrow \mid \frac{\partial \psi}{\partial x}\mid\leq R/\surd{2}~a.e.~and~\mid \frac{\partial \psi}{\partial y}\mid \leq R/\surd{2}~a.e.
{\nonumber}\end{equation}
\begin{equation}
So, \mid \int_{\Omega}\frac{\partial \psi_m}{\partial x}(v_m^2\frac{\partial \theta_m}{\partial y}-\frac{\partial \theta}{\partial y}v^2)\mid\\
\leq \int_{\Omega}\mid \frac{\partial \psi_m}{\partial x}\frac{\partial \theta_m}{\partial y}(v_m^2-v^2)\mid+\int_{\Omega}\mid \frac{\partial \psi_m}{\partial x}(\frac{\partial \theta_m}{\partial y}-\frac{\partial \theta}{\partial y})v^2 \mid{\nonumber}\end{equation}
\begin{equation}
\leq \frac{R^2}{2}\|v_m^2-v^2\|_2\mu(\Omega)^{1/2}+\frac{R}{\surd{2}}\|\frac{\partial v}{\partial y}\|_2\|\theta_m-\theta\|_2 ~\rightarrow~0~as~m~\rightarrow~\infty.
{\nonumber}\end{equation}
Similarly $\int_\Omega \mid \frac{\partial \psi_m}{\partial y}\frac{\partial \theta_m}{\partial x}v_m^2-\frac{\partial \psi}{\partial y}\frac{\partial \theta}{\partial x}v^2\mid~\rightarrow ~0~as~m \rightarrow ~\infty.\mu(\Omega)$ denotes the Lebesgue measure of $\Omega$.
\end{proof}
\begin{lem} Problem P has at most one solution $(\psi,\theta), (1-\frac{CR}{\surd{2}}-\frac{CR_a}{2}) >0,\\(1-\frac{R_a}{2}-2\surd{2}RC) >0.$
\end{lem}
\begin{proof}Let $(\psi_1,\theta_1)~and~(\psi_2,\theta_2)$ be two different  solutions of the problem satisfying the above estimate.\\
  $-\int_{\Omega} \triangle \psi_1 v_1 - R_a \int_{\Omega}(\frac{\partial \theta}{\partial x})v_1 - \int_{\Omega}\triangle \theta_1 v_2 + \int_{\Omega} J(\psi_1,\theta_1)  v_2 -\int_{\Omega}f_1v_1-int_{\Omega}f_2v_2 = 0$.\\
   $-\int_{\Omega} \triangle \psi_2 v_1 - R_a \int_{\Omega}({\frac{\partial \theta}{\partial x}})v_1 - \int_{\Omega}\triangle \theta_2 v_2 + \int_{\Omega}J(\psi_2,\theta_2) v_2 -\int_{\Omega}f_1v_1\int_{\Omega}f_2v_2 = 0$ where $v_1,v_2~ \in~ H_0^1(\Omega).$\\$~$\\
Subtracting both the equations and taking $\psi_1-\psi_2=\bar \psi=v_1,\theta_1-\theta_2=\bar \theta =v_2$ as test functions we have
\begin{equation}
\| \nabla \bar \psi\|_2^2 + \|\nabla \bar  \theta\|_2^2  -R_a\int_{\Omega}\frac{\partial {\bar \theta}}{\partial x}{{\bar \psi}} +
\int_{\Omega}[ J (\psi_1,\theta_1) - J (\psi_2,\theta_2)]\bar \theta =0.
\end{equation}
Now let us estimate the term

$\int_\Omega[ J (\psi_1,\theta_1) - J (\psi_2,\theta_2)]\bar \theta $.
\begin{equation}
[ J (\psi_1,\theta_1) - J (\psi_2,\theta_2)]\bar \theta {\nonumber}\end{equation}
\begin{equation}
=\frac{\partial \psi_1}{\partial x}\frac{\partial \theta_1}{\partial y}\bar \theta-\frac{\partial \psi_1}{\partial y}\frac{\partial \theta_1}{\partial x}\bar \theta-\frac{\partial \psi_2}{\partial x}\frac{\partial \theta_2}{\partial y}\bar \theta+\frac{\partial \psi_2}{\partial y}\frac{\partial \theta_2}{\partial x}\bar \theta.
{\nonumber}\end{equation}
 \begin{equation}=\frac{\partial(\bar \psi+\psi_2)}{\partial x}\frac{\partial(\bar \theta +\theta_2)}{\partial y}\bar \theta-
\frac{\partial(\bar \psi+\psi_2)}{\partial y}\frac{\partial(\bar \theta +\theta_2)}{\partial x}\bar \theta-\frac{\partial\psi_2}{\partial x}\frac{\partial \theta_2}{\partial y}\bar \theta+\frac{\partial \psi_2}{\partial y}\frac{\partial \theta_2}{\partial x}\bar \theta{\nonumber}\end{equation}
\begin{equation}
=\frac{\partial \bar \psi}{\partial x}\frac{\partial \bar \theta}{\partial y}\bar \theta+\frac{\partial \bar \psi}{\partial x}\frac{\partial \theta_2}{\partial y}\bar \theta+\frac{\partial \psi_2}{\partial x}\frac{\partial \bar \theta}{\partial y}\bar \theta-\frac{\partial \bar \psi}{\partial y}\frac{\partial \bar \theta}{\partial x}\bar \theta-\frac{\partial \bar \psi}{\partial y}\frac{\partial \theta_2}{\partial x}\bar \theta-\frac{\partial \bar \theta}{\partial x}\frac{\psi_2}{\partial y}\bar \theta\\
{\nonumber}\end{equation}
Our obtained solutions lie in the ball of radius R by Lemma$1$. So,
\begin{equation}
\|\nabla \psi_1\|_2,\|\nabla \psi_2\|_2~ \leq R
{\nonumber}\end{equation}
\begin{equation}
\|\nabla \theta_1\|_2,\|\nabla \\theta_2\|_2~ \leq R
{\nonumber}\end{equation}

 \begin{equation}
 \Rightarrow \mid \frac{\partial \psi_i}{\partial x}\mid\leq R/\surd{2}~a.e.~and~\mid \frac{\partial \psi_i}{\partial y}\mid \leq R/\surd{2}~a.e.~\forall~i=1,2
{\nonumber}\end{equation}
\begin{equation}
 \mbox{Similarly,} \mid \frac{\partial \theta_i}{\partial x}\mid\leq R/\surd{2}~a.e.~and~\mid \frac{\partial \theta_i}{\partial y}\mid \leq R/\surd{2}~a.e.~\forall~i=1,2
{\nonumber}\end{equation}

\begin{equation}
\|\frac{\partial \bar \psi}{\partial x}\|_2 \leq \|\frac{\partial \psi_1}{\partial x}\|_2 + \|\frac{\partial \psi_2}{\partial x}\|_2 \leq \surd{2}R ~a.e.{\nonumber}\end{equation}
Therefore by using Cauchy-Schwartz and Poincare inequality we have
\begin{equation}
\int_{\Omega} \mid \frac{\partial {\bar \psi}}{\partial x} \frac{\partial {\bar \theta}}{\partial y} \bar \theta \mid \leq \surd{2} R C \|\nabla {\bar \theta}\|_2^2{\nonumber}\end{equation}
\begin{equation}
\int_{\Omega}\mid \frac{\partial \bar \psi}{\partial x}\frac{\partial \theta_2}{\partial y}\bar \theta\mid \leq \frac{CR}{\surd{2}} \|\nabla \bar \psi \|_2\|\bar \theta \|_2 \leq \frac{CR}{2\surd{2}}(\| \nabla  \bar \psi\|_2^2+\| \nabla \bar \theta \|_2^2){\nonumber}\end{equation}
\begin{equation}
\int_{\Omega}\mid \frac{\partial \psi_2}{\partial x}\frac{\partial \bar \theta}{\partial y}\bar \theta \mid \leq \frac{CR}{\surd{2}} \| \nabla \bar \theta \|_2^2{\nonumber}\end{equation}
\begin{equation}
\int_{\Omega}\mid \frac{\partial \bar \psi}{\partial y}\frac{\partial \theta_2}{\partial x}\bar \theta \mid \leq \frac{RC}{2 \surd{2}}(\| \nabla \bar \psi\|_2^2+\| \nabla \bar \theta \|_2^2) \leq \frac{RC}{\surd{2}}\|\nabla \bar\theta \|_2^2.{\nonumber}\end{equation}
After further estimation and using (3) and Holder's inequality we remain with
\begin{equation}
(1-\frac{CR}{\surd{2}}-\frac{CR_a}{2})\|{\nabla \bar \psi }\|_2^2 + (1-\frac{R_a}{2}-2\surd{2}RC)\|{\nabla \bar\theta}\|_2^2 \leq 0 ~a.e..{\nonumber}\end{equation}
By given conditions using $(1-\frac{CR}{\surd{2}}-\frac{CR_a}{2}),(1-\frac{R_a}{2}-2\surd{2}RC) >0$ we have  $\bar \psi =0 ~and~\bar\theta=0 $ that means  solution is unique.
\end{proof}
\end{proof}
\textbf{Stability estimate:}
 The solution is continuously dependent on the source term under the conditions $(\frac{1}{2}-\frac{CR}{2\surd{2}}-\frac{CR_a}{2}),(\frac{1}{2}-\frac{R_a}{2}-\frac{3}{2\surd{2}}RC) >0$. \\
From the previous step the system has a unique solution $(\psi,\theta)$ under the above mentioned conditions. Now taking $(\psi,\theta)$ as the test function in the weak formulation we have

\begin{equation}
\|{\nabla \psi }\|_2^2 + \|{\nabla \theta}\|_2^2-\frac{R_a}{2}(C\|{\nabla  \psi }\|_2^2+\|{\nabla \theta}\|_2^2)-\frac{CR}{\surd{2}}\|{\nabla \psi }\|_2^2-2\surd{2}RC\|{\nabla \theta}\|_2^2
 \leq C(\|f_2\|_1\|\nabla \psi\|_2+\|f_2\|_2\|\nabla \theta\|_2)
\nonumber\end{equation}
\begin{equation}
(\frac{1}{2}-\frac{CR}{2\surd{2}}-\frac{CR_a}{2})\|{\nabla \psi }\|_2^2+(\frac{1}{2}-\frac{R_a}{2}-\frac{3}{2\surd{2}}RC)\|{\nabla \theta}\|_2^2 \leq \frac{C^2}{2}({\|f_1\|_2^2+\|f_2\|_2^2}).
\nonumber\end{equation}
$\Rightarrow (\frac{1}{2}-\frac{CR}{2\surd{2}}-\frac{CR_a}{2})\|{\nabla \psi }\|_2^2+(\frac{1}{2}-\frac{R_a}{2}-\frac{3}{2\surd{2}}RC)\|\nabla \theta\|_2^2 \leq \frac{C^2}{2}({\|f_1\|_2^2+\|f_2\|_2^2})$
\\$~$\\\big[Using the estimate from Lemma $2$ we have
\begin{equation}
 \mid \frac{\partial \psi}{\partial x}\mid\leq R/\surd{2}~a.e.,~\mid \frac{\partial \psi}{\partial y}\mid \leq R/\surd{2}~a.e.~and~
 \mid \frac{\partial \theta}{\partial x}\mid\leq R/\surd{2}~a.e.~,~\mid \frac{\partial \theta}{\partial y}\mid \leq R/\surd{2}~a.e.~
{\nonumber}\end{equation}

So, we have the estimation of the nonlinear term as
\begin{equation}
\int_{\Omega}\mid J(\psi,\theta)\theta \mid ~\leq \frac{CR}{2\surd{2}}\|\nabla \psi\|_2^2+\frac{3}{2\surd{2}}RC\|\nabla \theta\|_2^2
\nonumber \end{equation}]\\
Let $L=min\{(\frac{1}{2}-\frac{CR}{2\surd{2}}-\frac{CR_a}{2}),(\frac{1}{2}-\frac{R_a}{2}-\frac{3}{2\surd{2}}RC)\}$. We have assumed that $L>0$.\\
$\Rightarrow \|{\nabla \psi }\|_2^2+\|\nabla \theta\|_2^2 \leq \frac{C^2}{2L}(\|f_1\|_2^2+\|f_2\|_2^2)$\\
$\Rightarrow \|{\nabla \psi }\|_2,\|\nabla \theta\|_2 \leq \frac{C}{2\surd{L}}{\surd{(\|f_1\|_2^2+\|f_2\|_2^2)}}$

\newpage
\section{Finite Element Analysis}\label{sec:3}
\label{}
\textbf{Finite Element Formulation:} Let $\tau_h$ be the finite element partition of the domain $\Omega$ with mesh size h i.e.
$\Omega=\underset{k \in \tau_h}{\cup K }$. Define the finite element spaces corresponding to the Velocity, and temperature respectively as\\ $V_h=\{v~\in (H^1(\Omega))^2:V|_K \in V_r(K),\forall K \in \tau_h \}$ $V_r $ is the space of polynomial of degree $r\geq 1$.\\
\textbf{Interpolation error estimate:} We assume that the projection operator $\Pi^1_h:H^1(\Omega) \rightarrow V_h$  which satisfies the following interpolation error estimate
\begin{equation}
\|(v-\Pi^1_h v)\|_{s,K} \leq Ch^{t-s}\|v\|_{t,K}~s=0,1,~1\leq t\leq (r+1).\end{equation}
\begin{thm}
Let $(\psi,\theta)$ be the unique solution of the system of PDE $(1)-(2)$  under the conditions mentioned in theorem $1$. Let the discrete problem has a unique solution $(\psi_h,\theta_h)$ under the same conditions.\\ Then $|\Pi_h^1\psi-\psi_h|_{1,\Omega} \leq \frac{1}{2}{[\|\nabla(\Pi_h^1\psi-\psi) \|_2+(R_a^2+R+1)\|\nabla (\Pi_h^1\delta-\delta)\|_2]}$\\$|\Pi_h^1\psi-\psi_h|_{1,\Omega} \leq \frac{c}{2}{[\|\nabla(\Pi_h^1\delta-\delta) \|_2+(R_a^2+R+1)\|\nabla (\Pi_h^1\delta-\delta)\|_2]}$, $c=\frac{1}{\surd((1-2\surd{2}R))}$.

\end{thm}
\textbf{Discrete Problem:}Corresponding weak formulation is
\begin{equation}
\int_{\Omega}\nabla \psi_h\cdot\nabla v_h- R_a\int_{\Omega}\frac{\partial \theta_h}{\partial x}v_h + \int_{\Omega}\nabla \theta\cdot\nabla \phi_h + \int_\Omega  J \phi_h - \int_{\Omega}(f_1 v_h+f_2 \phi_h)=0\nonumber\end{equation}where $v_h,\phi_h~\in~H_0^1(\Omega), \Omega=(0,1)\times (0,1)$.\\
\textbf{Error Equation:}\begin{equation}
\int_{\Omega}\nabla(\psi- \psi_h)\cdot\nabla v_h- R_a\int_{\Omega}\frac{\partial( \theta-\theta_h)}{\partial x}v_h + \int_{\Omega}\nabla (\theta-\theta_h)\cdot\nabla \phi_h + \int_\Omega  [J(\psi,\theta)-J(\psi_h,\theta_h)]\phi_h=0
\nonumber\end{equation}
Let $\xi=\psi-\Pi_h \psi,~\xi_h=\psi_h-\Pi_h \psi,~\delta=\theta-\Pi_h \theta,~\delta_h=\theta_h-\Pi_h \theta.$\\
Taking $v_h=\xi_h,~ \phi_h=\delta_h$ in the equation (5) we have\\
\begin{equation}
\int_{\Omega}\nabla(\psi- \psi_h)\cdot\nabla \xi_h- R_a\int_{\Omega}\frac{\partial( \theta-\theta_h)}{\partial x}\xi_h + \int_{\Omega}\nabla (\theta-\theta_h)\cdot\nabla \delta_h + \int_\Omega  [J(\psi,\theta)-J(\psi_h,\theta_h)]\delta_h=0.
\nonumber\end{equation}
\begin{equation}
\int_{\Omega}\nabla(\xi- \xi_h)\cdot\nabla \xi_h- R_a\int_{\Omega}\frac{\partial( \delta-\delta_h)}{\partial x}\xi_h + \int_{\Omega}\nabla (\delta-\delta_h)\cdot\nabla \delta_h + \int_\Omega  [J(\psi,\theta)-J(\psi_h,\theta_h)]\delta_h=0.
\nonumber\end{equation}
$\Rightarrow$
\begin{equation}
\begin{split}
\|\nabla \xi_h\|_2^2+\|\nabla \delta_h\|_2^2 &= \int_{\Omega}\nabla \xi \cdot \nabla \xi_h-R_a\int_{\Omega}(\delta-\delta_h)\xi_h
+\int_\Omega  [J(\psi,\theta)-J(\psi_h,\theta_h)]\delta_h+\int_\Omega \nabla \delta \cdot \nabla \delta_h\\
&\leq \int_{\Omega}\nabla \xi \cdot \nabla \xi_h
+\int_\Omega  [J(\psi,\theta)-J(\psi_h,\theta_h)]\delta_h+\int_\Omega \nabla \delta \cdot \nabla \delta_h\\
&\leq \frac{1}{2}(\|\nabla \xi\|_2^2+\|\nabla \xi_h\|_2^2)+\frac{1}{2}(\surd{2}R+1)^{2}\|\nabla \delta\|_2^2+\frac{1}{2}\|\nabla \delta_h\|_2^2+\surd{2}R\|\nabla \delta_h\|_2^2.
\end{split}
\nonumber\end{equation}

\begin{equation}
\Rightarrow\|\nabla\xi_h\|_2^2+(1-2\surd{2}R)\|\nabla\|_2^2 \leq [\|\nabla \xi\|_2^2+(1+\surd{2}R)^{2}\|\nabla \delta\|_2^2]
\end{equation}
Then using the inequality $ a^2+b^2\leq(a+b)^2$ in (7) we have\\
 \begin{equation}\|\nabla(\Pi_h^1\psi-\psi_h)\|_{2} \leq \frac{1}{2}{[\|\nabla(\Pi_h^1\psi-\psi) \|_2+(1+\surd{2}R)\|\nabla (\Pi_h^1\delta-\delta)\|_2]}\end{equation}
\begin{equation}
\|\nabla(\Pi_h^1\psi-\psi_h)\|_{2} \leq \frac{c}{2}{[\|\nabla(\Pi_h^1\delta-\delta) \|_2+(1+\surd{2}R)\|\nabla (\Pi_h^1\delta-\delta)\|_2]}\end{equation}
where $c=\frac{1}{\surd((1-2\surd{2}R))}$,
 with appropriately small source term.\\$~$\\
\textbf{Remark $1$:}
The above proof holds for properly chosen finite element space and interpolation operator.\\
\textbf{Remark $2$:}
Observing the RHS of (3.3) and (3.4) and using (3.1) we have order of convergence is 1 for both velocity and temperature.

\section{Numerical Experiment:}
Here we have taken the test problem with the strong form written below
\begin{equation}
\triangle\psi = -R_a\frac{\partial \theta}{\partial x} +f_1\nonumber~~~~~ ~~~~~~~~~~~~~~~(1)
\end{equation}
\begin{equation}
J(\psi,\theta) = \triangle\theta+f_2, \nonumber~~~~~~~~~~~~~~~~~~~~~~~(2)
\end{equation},$\psi,\theta~\in~C^2(\Omega), \Omega=(0,1)\times (0,1)$ with the exact solutions\\ $\psi=2x^2(x-1)^2 y(y-1)(2y-1), \theta =(-2)y^2(y-1)^2 x(x-1)(2x-1)$, \\which vanishes on the boundary of the domain $\Omega$, with compatible $f_1,~f_2.$ 

Here we have used FEM++ software for numerical computation. Contours representing the computed solution solutions corresponding to velocity and temperature of the above problem are depicted in Figures given below with different values of $R_a$ .\\ 

\textbf{Algorithm for FEM computation:}\\
Let us denote the pair $(\psi,\theta)=u$.\\
Step1: Choose the initial vector $u_0=(\psi_{0},\theta_{0})~\in~\mathbb R^{2}$.\\
Step2: Choose $\epsilon=1e-08;$\\
int $i=0$;\\
while \\
$\{\|w_i\|_2<\epsilon\}$
then\\
begin\\
solve\\
(a)$DF(u_i)w_i=F(u_i);$~~
(b)$u_{i+1}=u_i-w_i;$\\
break  if $\|w_i\|~<\epsilon.$\\$~$\\
where $DF$ and F satisfy
$F(u+\delta)=F(u)+DF(u)\delta+o(\delta).$\\Here F and $DF$ are given as follows\\
$F(u)= F(\psi,\theta)=\int_{\Omega}(\nabla \psi \cdot \nabla v+\nabla \theta \cdot \nabla \tau+J(\psi,\theta)\tau-R_a(\frac{\partial \theta}{\partial x})v-f_1 v-f_2 \tau)$\\$~$\\
$Df(\psi,\theta)(\delta \psi,\delta \theta)=\int_{\Omega}((\nabla (\delta\psi) \cdot \nabla v+\nabla (\delta \theta )\cdot \nabla \tau-R_a(\frac{\partial \theta}{\partial x})v)\\+\int_{\Omega}(\frac{\partial(\delta \psi)}{\partial x}\frac{\partial \theta}{\partial y}+\frac{\partial \psi}{\partial x}\frac{\partial(\delta \theta)}{\partial y}-\frac{\partial(\delta \psi)}{\partial y}\frac{\partial \theta}{\partial x}-\frac{\partial \psi}{\partial y}\frac{\partial(\delta \theta)}{\partial x})\tau-\int_{\Omega}(f_1 v+f_2 \tau).$\\$~$\\
$i=i+1$\\
end



\bibliographystyle{amsplain}

\begin{thebibliography}{n} 

\bibitem{Bej13}
 Adrian Bejan,\emph{ Convection Heat Transfer}, John wiley $\&$ sons, 2013.

\bibitem{Cer12}
AyIc Cesmelioglu and Beatrice Riviere, \emph{ Existence of a weak solution for the
fully coupled Navier-Stokes/Darcy-transport problem}, Journal of Differential
Equations, 252(7):4138-4175, 2012.

\bibitem{Chm86}
 Michel Chipot and Mitchell Luskin, \emph{ Existence and uniqueness of solutions
to the compressible reynolds lubrication equation}, SIAM Journal on
Mathematical Analysis, 17(6):1390-1399, 1986.

\bibitem{Gir12}
 Vivette Girault and Pierre-Arnaud Raviart, \emph{ Finite element methods for Navier-
Stokes equations: theory and algorithms}, volume 5. Springer Science $\&$ Business
Media, 2012.

\bibitem{Gir09}
 Vivette Girault and Beatrice Riviere, \emph{ Dg approximation of coupled Navier-
Stokes and Darcy equations by Beaver-Joseph-Saffman interface condition},
SIAM Journal on Numerical Analysis, 47(3):2052-2089, 2009.

\bibitem{Kes88}
 S Kesavan. \emph{Functional Analysis and Applications}, Wiley, 1988.
 
 \bibitem{Xsh05}
 Meng Xu and Shulin Zhou, \emph{ Existence and uniqueness of weak solutions for
a generalized thin film equation. Nonlinear Analysis: Theory, Methods$ \&$
Applications}, 60(4):755-774, 2005.

\end{thebibliography}

\end{document}